%% file: mainIEEE.tex
\documentclass[conference]{IEEEtran}
\usepackage{cite}
\usepackage{amsmath,amssymb,amsfonts}
\usepackage{algorithm,algorithmic}
\usepackage{graphicx}
\usepackage{textcomp}
\usepackage{xcolor}

\usepackage{amsthm}
\usepackage{hyperref}

\newtheorem{theorem}{Theorem}

\newtheorem{fact}{Fact}

\newtheorem{definition}{Definition}

\newtheorem{lemma}{Lemma}

\newtheorem{proposition}{Proposition}
\newtheorem{remark}{Remark}

\numberwithin{equation}{section}

\renewcommand{\thefootnote}{\arabic{footnote}}

\newcommand{\MF}[1]{\textcolor{black}{#1}}

\newcommand{\calN}{\ensuremath{\mathcal{N}}}

\newcommand{\calQ}{\ensuremath{\mathcal{Q}}}

\newcommand{\bbT}{\ensuremath{\mathbb{T}}}

\newcommand{\norm}[1]{\left\|{#1}\right\|}
\newcommand{\abs}[1]{\left|{#1}\right|}
\newcommand{\ceil}[1]{\lceil{#1}\rceil}

\newcommand{\set}[1]{\left\{{#1}\right\}}

\newcommand{\est}[1]{\widehat{#1}}
\newcommand{\expec}{\ensuremath{\mathbb{E}}}
\newcommand{\matR}{\ensuremath{\mathbb{R}}}

\newcommand{\prob}{\ensuremath{\mathbb{P}}}

\newcommand{\rmi}{\iota}

\newcommand{\ghat}{\est{g}}

\newcommand{\indic}{\ensuremath{\mathbf{1}}} 

\newcommand{\gest}{\ensuremath{\est{g}}}

\newcommand{\ftil}{\ensuremath{\widetilde{f}}}

\newcommand{\htil}{\ensuremath{\widetilde{h}}}
\newcommand{\Wst}{\ensuremath{W^{*}}}
\newcommand{\Kmin}{\ensuremath{K_{\min}}}
\newcommand{\Kmax}{\ensuremath{K_{\max}}}
\newcommand{\expn}[1]{e^{{-{#1}\pi^2\sigma^2}}}
\newcommand{\expp}[1]{e^{{{#1}\pi^2\sigma^2}}}
\newcommand{\bias}{\text{Bias}}

\newcommand{\holdclass}{\ensuremath{C^{l,\alpha}([0,1],M)}}
\newcommand{\holdclassTwoPi}{\ensuremath{C^{l,\alpha}([0,1],2\pi M)}}

\newcommand{\holdclassnew}{\ensuremath{C^{l,\alpha}([0,1],M')}}

\newcommand\blfootnote[1]{%
  \begingroup
  \renewcommand\thefootnote{}\footnote{#1}%
  \addtocounter{footnote}{-1}%
  \endgroup
}

\def\BibTeX{{\rm B\kern-.05em{\sc i\kern-.025em b}\kern-.08em
    T\kern-.1667em\lower.7ex\hbox{E}\kern-.125emX}}
\begin{document}

\title{Recovering H\"older smooth functions from noisy modulo samples
}

\author{
\IEEEauthorblockN{ Micha\"el Fanuel}
\IEEEauthorblockA{\textit{Univ. Lille, CNRS, Centrale Lille, UMR 9189,} \\
\textit{CRIStAL}\\
F-59000 Lille, France \\
michael.fanuel@univ-lille.fr}
\and
\IEEEauthorblockN{ Hemant Tyagi$^*$}
\IEEEauthorblockA{\textit{Inria, Univ. Lille, CNRS, UMR 8524, } \\
\textit{Laboratoire Paul Painlev\'{e}}\\
F-59000 Lille, France \\
hemant.tyagi@inria.fr}
}

\maketitle

\begin{abstract}
\input{abstract}
\end{abstract}


\input{introduction}
\input{problem_setup}
\input{notations}
\input{denoise_mod1_local_poly}
\input{put_together}

\input{simulations}

\section*{Acknowledgment}
M.F. acknowledges support from ERC grant \textsc{Blackjack} (ERC-2019-STG-851866, PI: R. Bardenet). 
\bibliography{references}
\bibliographystyle{./IEEEtran}

\newpage
\onecolumn
\input{holder_smoothness}

\end{document}

%% file: abstract.tex
In signal processing, several applications involve the recovery of a function given noisy modulo samples. The setting considered in this paper is that the samples corrupted by an additive Gaussian noise are wrapped due to the modulo operation. Typical examples of this problem arise in phase unwrapping problems or in the context of self-reset analog to digital converters. We consider a fixed design setting where the modulo samples are given on a regular grid. Then, a three stage recovery strategy is proposed to recover the ground truth signal up to a global integer shift. The first stage denoises the modulo samples by using local polynomial estimators. In the second stage, an unwrapping algorithm is applied to the denoised modulo samples on the grid. Finally, a spline based quasi-interpolant operator is used to yield an estimate of the ground truth function up to a global integer shift.  For a function in H\"older class, uniform error rates are given for recovery performance with high probability. This extends recent results obtained by Fanuel and Tyagi for Lipschitz smooth functions wherein $k$NN regression was used in the denoising step.

%% file: introduction.tex
\section{Introduction}
Various signal processing applications deal with noisy modulo samples of a signal. A typical example is ``phase unwrapping'' where the data is often available in the form of modulo $2\pi$ samples and which arises during the estimation of the depth map of a terrain (e.g., \cite{graham_insar,zebker86}), or in the context of biomedical applications~\cite{hedley92,laut73}. Another application involves self-reset analog to digital converters~\cite{Kester,RHEE03,yamaguchi16} in a context where the reset counts are not used but only modulo samples are available. This is closely related to the works~\cite{bhandari17,sparse_unlimsin18} where a folding of the signal is deliberately injected.\blfootnote{$^*$ The authors are listed in alphabetical order.}

%% file: problem_setup.tex
\section{Problem setup}
Let $f:[0,1]^d\to \mathbb{R}$ be an unknown function, we assume that we are given noisy modulo $1$ samples of $f$ on a uniform grid, i.e., 
\begin{equation} \label{eq:fmod1_model}
    y_i = (f(i/n) + \eta_i) \bmod 1; \quad i = 1,\dots, n. 
\end{equation}
Here, $\eta_i \sim \calN(0, \sigma^2)$ are i.i.d Gaussian noise samples, with $\sigma$ denoting the noise level. It will be useful to denote $x_i = i/n$ from now on.  We will assume that $f \in \holdclass$ which denotes the H\"older class of functions, defined below.
\begin{definition} \label{def:hold_class}
For $l \in \mathbb{N}_0$, $\alpha \in (0,1]$ and $M > 0 $, the H\"older class $\holdclass$ consists of $l$ times continuously differentiable functions $f:[0,1] \rightarrow \matR$ whose $l$th derivative $f^{(l)}$ satisfies
$    \abs{f^{(l)}(x) - f^{(l)}(y)} \leq M \abs{x-y}^{\alpha} \quad \forall  x,y \in [0,1]
$. 
Moreover, $\beta:= l + \alpha > 0$ denotes the smoothness of this class.
\end{definition}
Given $(y_i)_{i \in [n]}$, our aim is to obtain an estimate $\est{f}: [0,1] \rightarrow \matR$ of $f$. Clearly, we can only hope to recover $f$ up to a global integer shift. To this end, we will follow a three-stage strategy considered recently in \cite{fanuel20} for estimating Lipschitz $f$. We recall the steps below.
\begin{enumerate}
    \item ({\bf Stage 1:} Denoise modulo samples) This involves mapping the noisy modulo samples onto the unit complex circle $\mathbb{T}_1$ as 
    \begin{equation} \label{eq:z_noisy}
    z_i = \exp(\rmi 2\pi y_{i}) =  h_i         \exp(\rmi 2\pi \eta_i), 
\end{equation}
    with $h_i = h(x_i) = \exp(\rmi 2\pi f(x_i))$ denoting the clean modulo samples with $i = 1,\dots, n$, and where $h:[0,1] \to \bbT_1$ is defined as $h(x) = \exp(\rmi 2\pi f(x))$. The idea in \cite{fanuel20} was to note that if $f$ is Lipschitz smooth, then it implies that $h$ is also Lipschitz. This motivated a nearest neighbor based denoising procedure for uniformly estimating $h_i$ via $\est{h}_i \in \bbT_1$, for all $i$. That in turn lead to  estimates $\est{f}_i \bmod 1 = \frac{1}{2\pi} \arg(\est{h}_i)$ of $f_i \bmod 1 = \frac{1}{2\pi} \arg(h_i)$ with a uniform bound 
        $$d_w(\est{f}_i \bmod 1, f_i \bmod 1) \lesssim \delta(n), \ \forall i \in [n].$$

    \item ({\bf Stage 2:} Unwrap denoised modulo samples) Given the estimates $\est{f_i \bmod 1}$ from the first stage, we then perform an unwrapping procedure reminiscent of Itoh's method from phase unwrapping. Denoting $L$ to be the Lipschitz constant of $f$, if $\delta(n) + \frac{L}{n} \lesssim 1$, then this procedure returns estimates $\ftil(x_i)$ such that (see \cite[Lemma 5]{fanuel20})
    \begin{equation*}
        \abs{\ftil(x_{i}) + q^{\star} - f(x_{i})} \lesssim \delta(n), \quad \forall i \in [n].
    \end{equation*}

    \item ({\bf Stage 3:} Obtain $\est{f}$ via quasi-interpolants) Given the estimates  $(\ftil(x_i))_{i \in [n]}$ from the second stage, the estimate $\est{f} : [0,1] \rightarrow \matR$ is finally obtained by applying a suitable quasi-interpolant operator on these estimates. Additionally, one can readily show the error bound (see \cite[Theorem 5]{fanuel20})
    \begin{equation} \label{eq:fest_bound}
            \norm{\est{f} + q^* - f}_{\infty} \leq C_1 n^{-1} + C_2 \delta(n)
    \end{equation}
    where $C_1, C_2$ are absolute constants, and $q^* \in \mathbb{Z}$ is some integer.
\end{enumerate}
The setup in \cite{fanuel20} assumes $f$ to be Lipschitz continuous which results in $\delta(n) = O((\frac{\log n}{n})^{1/3})$ for univariate functions\footnote{In \cite{fanuel20} the rate $O((\frac{\log n}{n})^{1/(d+2)})$ was derived for $d$-variate $f$.}. For the estimation of $f$, this leads to the $L_{\infty}$ error rate $(\frac{\log n}{n})^{1/3}$ which matches the optimal $L_{\infty}$ rate for estimating univariate Lipschitz functions for the model \eqref{eq:fmod1_model} without the modulo operation (see \cite{nemirovski2000topics}). 

\paragraph*{Our goal} We aim to extend this result to the more general setting where $f \in \holdclass$. We will show this by modifying the denoising procedure in Stage 1 where we will instead consider a local polynomial estimator of order $l$. Such estimators are classical in the nonparametric regression literature, we will adapt the analysis in the book of Tsybakov \cite[Chapter 1]{tsybakov08} to our setting and show that $\delta(n) = O((\frac{\log n}{n})^{\frac{\beta}{2\beta + 1}})$ where $\beta := l+\alpha > 0$. Due to \eqref{eq:fest_bound}, this will then imply the $L_{\infty}$ rate $(\frac{\log n}{n})^{\frac{\beta}{2\beta + 1}}$ for estimating $f$ which matches the optimal $L_{\infty}$ rate for estimating functions lying in $\holdclass$, for the model  \eqref{eq:fmod1_model} without the modulo operation (see \cite{nemirovski2000topics}). 

%% file: notations.tex
\paragraph*{Notations}
We denote $[n] = \{1,\dots , n\}$. The imaginary unit is $\rmi$ such that $ \rmi^2 = -1$.  The product of unit circles is denoted by $\mathbb{T}_n = \mathbb{T}_1 \times \cdots \times \mathbb{T}_1$ with $\mathbb{T}_1 := \set{u \in \mathbb{C}: \abs{u} = 1}$. We define the projection of $u \in \mathbb{C}^n$ on $\mathbb{T}_n$ as 
\begin{equation*}
    \left(\frac{u}{|u|}\right)_i = \begin{cases}
    \frac{u_i}{|u_i|} \text{ if } u_i\neq 0,  \\
    1 \text{ otherwise, }
    \end{cases}
\end{equation*}
for all  $i\in[n]$. The commonly used wrap around metric $d_w:[0,1) \rightarrow [0,1/2]$ is  defined as
$
 d_w(t,t') := \max \set{\abs{t-t'}, 1-\abs{t-t'}}.
$
For $1 \leq p \leq \infty$, $\norm{x}_p$ is the $\ell_p$ norm of $x \in \mathbb{C}^n$. 
For any $a,b \geq 0$, we write $a \lesssim b$ if there is $C > 0$ such that $a \leq C b$. Moreover, we write $a\asymp b$ if $a \lesssim b$ and $b \lesssim a$.

%% file: denoise_mod1_local_poly.tex
\section{Denoising modulo samples via local polynomial estimators}
Denoting $h_R(x), h_I(x)$ to be the real and imaginary parts of $h(x) = \exp(\rmi 2\pi f(x))$ for any $x \in [0,1]$, a crucial observation that we  use is that if $f \in \holdclass$, and if $f^{(\ell)}$ are uniformly bounded for all $0\leq \ell \leq l$, then it implies $h_R, h_I \in \holdclassnew$ for some constant $M' > 0$. 
%
%
%
\begin{proposition} \label{prop:smoothness_h}
Suppose $f \in \holdclass$ for some $l \in \mathbb{N}_0,\alpha \in (0,1]$ and $M > 0$. Further, assume that $\|f^{(\ell)}\|_{\infty}\leq \kappa$ for some $\kappa>0$ and for all integers $0\leq \ell \leq l$. Then there exists $M' > 0$ depending only on $l, M, \kappa$ such that for the functions $h_R(x) = \cos(2\pi f(x))$ and $h_I(x) = \sin(2\pi f(x))$, we have that $h_R, h_I \in \holdclassnew$.
\end{proposition}
The proof of Proposition \ref{prop:smoothness_h} and all other results in the paper are outlined in the appendix. 
Our estimator $\est{h}(x)$ of $h(x) = \exp(\rmi 2\pi f(x))$  will constructed via a local polynomial estimator of order $l$, namely LP($l$). 
Before introducing the estimator, let us first define some additional quantities following the notation in  \cite[Section 1.6]{tsybakov08}.
\begin{itemize}
    \item $K: \matR \rightarrow \matR$ denotes a kernel, and $b > 0$ its bandwidth. 
    
    \item For any $u \in \matR$ and integer $l \geq 0$, 
    \begin{equation*}
        U(u) = (1,u,u^2/2!,\dots,u^l/l!)^{\top} \in \matR^{l+1}. 
    \end{equation*}
    
    \item For any $x \in [0,1]$, $B_{nx} \in \matR^{(l+1) \times (l+1)}$ denotes the matrix
    \begin{equation*}
        B_{nx} = \frac{1}{nb} \sum_{i=1}^{n} U\left(\frac{x_i-x}{b}\right) U^{\top}\left(\frac{x_i-x}{b}\right) K\left(\frac{x_i-x}{b} \right).
    \end{equation*}
\end{itemize}
\paragraph{Local polynomial estimator}  Denoting $z = (z_1,\dots,z_n)^{\top} \in \bbT_n$ with $z_i \in \bbT_1$ as in \eqref{eq:z_noisy}, we first compute
\begin{align*}
    \est{\theta}_{n}(x) &= \arg\min_{\theta \in \mathbb{C}^{l+1}} \sum_{i=1}^n \left|z_{i} - \theta^{\top} U\left(\frac{x_i-x}{b}\right) \right|^2 K\left(\frac{x_i-x}{b} \right). 
\end{align*}
Denote $\est{\theta}_{n,R}, \est{\theta}_{n,I} \in \matR^n$ to be the real and imaginary parts of $\est{\theta}_{n}$.
Then, we obtain a preliminary estimate
\begin{equation} \label{eq:htil_lp_est}
    \htil(x) = U^{\top}(0) \est{\theta}_{n,R}(x) + \iota U^{\top}(0) \est{\theta}_{n,I}(x) =: \htil_{R}(x) + \iota \htil_{I}(x). 
\end{equation}
Here, $\htil_{R}(x), \htil_{I}(x)$ are LP($l$) estimators of $h_R(x)$ and $h_I(x)$ respectively (see \cite[Definition 1.8]{tsybakov08}).
\paragraph{Projection step} Next, we project $\htil(x)$ onto $\bbT_1$ to obtain $\est{h}(x)$, and hence $\est{f(x) \bmod 1}$, as
\begin{equation} \label{eq:hest_proj}
    \est{h}(x) := \frac{\htil(x)}{\abs{\htil(x)}}, \quad \est{f(x) \bmod 1} = \frac{1}{2\pi} \arg(\est{h}(x)).
\end{equation}
\begin{remark}
If $B_{nx} \succ 0$ then $\htil(x)$ can be uniquely written as   
\begin{equation} \label{eq:htil_uniq_exp}
    \htil(x) = \htil_{R}(x) + \iota \htil_{I}(x) = \sum_{i=1}^n z_i \Wst_{ni}(x), 
\end{equation}
where $\Wst_{ni}(x) \in \matR$ is given by \cite{tsybakov08}  
\begin{equation*}
    \Wst_{ni}(x) = \frac{1}{nb}  U^{\top}(0) B_{nx}^{-1} U \left(\frac{x_i-x}{b} \right)    K\left(\frac{x_i-x}{b} \right).
\end{equation*}
\end{remark}
Our aim is to show that $\abs{\est{h}(x_i) - h(x_i)}$ is uniformly bounded for all $i \in [n]$. To show this, we will need some preliminary tools from \cite[Section 1.6]{tsybakov08}.
\subsection{Preliminaries}
Firstly, we recall \cite[Proposition 1.12]{tsybakov08} which states that under mild assumptions, the LP($l$) estimator  reproduces polynomials of degree less than or equal to $l$.
\begin{proposition}[{\cite[Proposition 1.12]{tsybakov08}}] \label{prop:wst_reprod}
Let $x$ be such that $B_{nx} \succ 0$ and let $Q$ be a polynomial of degree $\leq l$. Then the LP($l$) weights $W^*_{ni}$ are such that 
$  \sum_{i=1}^n Q(x_i) W^*_{ni}(x) = Q(x).$
In particular, 
\begin{equation*}
  \sum_{i=1}^n W^*_{ni}(x) = 1, \quad \sum_{i=1}^n (x_i - x)^k W^*_{ni}(x) = 0 \ \text{ for } k=1,\dots,l.   
\end{equation*}
\end{proposition}
Next, we need to establish conditions under which $B_{nx} \succ 0$ for all $x \in [0,1]$. %
\begin{lemma}[{\cite[Lemma 1.5]{tsybakov08}}]  \label{lem:mineig_Bnx}
Suppose there exist $\Kmin > 0$ and $\Delta > 0$ such that
$    K(u) \geq \Kmin \indic_{\set{\abs{u} \leq \Delta}},$ $ \forall u \in \matR. 
$
Let $b = b_n$ be a sequence satisfying $b_n \rightarrow 0$ and $nb_n \rightarrow \infty$ as $n \rightarrow \infty$. Then there exist $\lambda_0, n_0 > 0$ such that 
$
    \lambda_{\min}(B_{nx}) \geq \lambda_0
$
for all $n \geq n_0$ and any $x \in [0,1]$.
\end{lemma}
Finally, we recall the following result from \cite[Lemma 1.3]{tsybakov08} which states useful properties for the weights $W^{*}_{ni}$.
\begin{lemma}[{\cite[Lemma 1.3]{tsybakov08}}] \label{lem:wst_props}
Suppose that the kernel $K$ has compact support in $[-1,1]$ and there exists a number $\Kmax < \infty$ such that $\abs{K(u)} \leq \Kmax$, $\forall u \in \matR$. Then under the assumptions of Lemma \ref{lem:mineig_Bnx}, and with $b \geq 1/(2n)$, the following is true.
\begin{enumerate}
    \item $\sup_{i,x} \abs{\Wst_{ni}(x)} \leq \frac{C_*}{nb}$, 
    
    \item $\sum_{i=1}^n \abs{\Wst_{ni}(x)} \leq C_*$,
    
    \item $\Wst_{ni}(x) = 0$ if $\abs{x_i - x} > b$,
\end{enumerate}
where $C_* = \frac{8 \Kmax}{\lambda_0}$.
\end{lemma}
%
%
%
\subsection{Analysis}
We now derive conditions so that with high probability, $\abs{\est{h}(x_i) - h(x_i)} = O((\frac{\log n}{n})^{\frac{\beta}{2\beta + 1}})$ holds for all $i \in [n]$. The main tool is the following lemma which details the bias-variance trade-off in the estimation error.
%
%
\begin{lemma} \label{lem:denoise_bias_var}
Assuming $h_R, h_I \in \holdclassnew$ and denoting $\beta = l + \alpha$, let $\htil_{R}$,$\htil_{I}$ be their respective LP($l$) estimators as in \eqref{eq:htil_lp_est}. Under the notation and conditions of Lemma's \ref{lem:mineig_Bnx} and \ref{lem:wst_props}, and assuming $nb \geq \log n$, it holds with probability at least $1-4n^{1-c}$ (for any $c \geq 2$) that the estimator $\est{h}$ in \eqref{eq:hest_proj} satisfies
\begin{equation} \label{eq:bias_var_bd}
    \abs{\est{h}(x_i) - h(x_i)} \leq q_1 b^{\beta} + q_2 \left(\frac{\log n}{nb} \right)^{1/2}, \quad \forall i \in [n],
\end{equation}
with $q_1 = \frac{4M' C_*}{l!}$,  $q_2 = 8c C_* A(\sigma)$ where $A(\sigma) := \sqrt{\expp{4}-1} + (1+\expp{2})$. 
\end{lemma}
%
%
%
%
%
We now easily obtain the following result by instantiating Lemma \ref{lem:denoise_bias_var} for the best choice of $b$. This is our main denoising result which we had set out to prove in this section.
\begin{theorem} \label{thm:main_denoise_mod}
Let $f \in \holdclass$ with $\|f^{(\ell)}\|_{\infty}\leq \kappa$ for some $\kappa>0$ and for all integers $0\leq \ell \leq l$. Consider the estimator $\est{h}(x)$ of $h(x) = \exp(\iota 2\pi f(x))$ as in \eqref{eq:hest_proj}, at any $x \in [0,1]$. Denoting $\beta = l + \alpha$, for any $c \geq 2$,  suppose that 
\begin{enumerate}
    \item there exist constants $\Kmin, \Kmax, \Delta > 0$ so that $$\Kmin \indic_{\set{\abs{u} \leq \Delta}} \leq K(u) \leq \Kmax\indic_{\set{\abs{u} \leq 1}}, \quad \forall u \in \matR;$$
    
    \item $b = b^* = (\frac{c A(\sigma) l!}{\beta M'})^{\frac{2}{2\beta+1}} (\frac{\log n}{n})^{\frac{1}{2\beta+1}}$ with $A(\sigma)$ as in Lemma \ref{lem:denoise_bias_var}, and $M' > 0$ depending only on $M,l,\kappa$;
    
    \item $n \geq n_0$, $\frac{n}{\log n} \geq (\frac{\beta M'}{c A(\sigma) l!})^{1/\beta}$ with $n_0$ as in Lemma \ref{lem:mineig_Bnx}. 
\end{enumerate}
Then with probability at least $1-4n^{1-c}$ we have for all $i \in [n]$,
\begin{align} 
    &\abs{\est{h}(x_i) - h(x_i)}\nonumber\\
    & \leq \left( \frac{32M' \Kmax}{l! \ \lambda_0 } \right)^{\frac{1}{2\beta+1}} \left( \frac{64c \Kmax A(\sigma)}{\lambda_0} \right)^{\frac{2\beta}{2\beta+1}}\times \nonumber\\
    & \left[(2\beta)^{-\frac{2\beta}{2\beta+1}} + (2\beta)^{\frac{1}{2\beta+1}} \right]
     \left(\frac{\log n}{n} \right)^{\frac{\beta}{2\beta+1}}\nonumber\\
    &=: \delta(n), \label{eq:h_denoise_fin_bd} 
\end{align}
where $\lambda_0$ is as in Lemma \ref{lem:mineig_Bnx}. Furthermore, if \eqref{eq:h_denoise_fin_bd} holds and if $\delta(n) \leq 2$, then
\begin{equation} \label{eq:modulo_denoise_fin_bd} 
    d_w(\est{f(x_i) \bmod 1}, f(x_i) \bmod 1) \leq \frac{\delta(n)}{4}, \quad \forall i \in [n]. 
\end{equation}
\end{theorem}
\begin{proof}
Denoting $F(b)$ to be the bound in \eqref{eq:bias_var_bd}, we easily verify that the global minimizer of $F$ is 
$b^* = \left(\frac{q_2}{2\beta q_1} \right)^{\frac{2}{2\beta + 1}} \left(\frac{\log n}{n} \right)^{\frac{1}{2\beta+1}}$ with $q_1, q_2$ as in Lemma \ref{lem:denoise_bias_var}.  
Clearly $b_n^* \rightarrow 0$ as $n \rightarrow \infty$, and $n b^*_n \geq \log n$ for the stated condition on $n/\log n$. Plugging $b = b^*$ in \eqref{eq:bias_var_bd} leads to the stated bound in \eqref{eq:h_denoise_fin_bd} after some simplifications. The bound in \eqref{eq:modulo_denoise_fin_bd} follows directly using \cite[Fact 4]{fanuel20}.
\end{proof}

%% file: put_together.tex
\section{Error rate for recovering $f$}
Given the denoised estimates of $f(x_i) \bmod 1$ for each $i$, we can now recover $f$ following the steps described in \cite{fanuel20}. Indeed, we first recover estimates $\ftil(x_i)$ of the samples $f(x_i) \bmod 1$ using the sequential unwrapping procedure outlined as Algorithm 2 in \cite{fanuel20}, for the univariate setting $d = 1$.  
Denoting $\gest(x_i) = \est{f(x_i) \bmod 1}$, we recall from \cite{fanuel20} that $\ftil(x_1) = \ghat(x_1)$ and
\begin{equation} \label{eq:seq_unwrap_proc}
     \ftil(x_i) = \ftil(x_{i-1}) + 
    \left\{
\begin{array}{rl}
d_i \ ; & \text{ if } \abs{d_i} < 1/2, \\
1 + d_i \ ; & \text{ if } d_i < -1/2, \\
-1 + d_i \ ; & \text{ if } d_i > 1/2.
\end{array} \right.
\end{equation}
for $d_i = \ghat(x_i) - \ghat(x_{i-1})$ with $i\geq 2$.
We also know that there exists $L > 0$ such that 
\begin{equation} \label{eq:lip_L}
 \abs{f(x)-f(y)} \leq L \abs{x-y}^{\min\set{\beta,1}},
\end{equation}
where $L = M$ if $\beta \leq 1$, and $L = \kappa$ otherwise. 
Then, following the steps in \cite[Lemma's 2.2 2.3]{fanuel20}, it is easy to verify that if \eqref{eq:modulo_denoise_fin_bd} holds along with the condition $\delta(n) + \frac{2L}{n^{\min\set{\beta,1}}} < 1$  then for some integer $q^*$, 
\begin{equation} \label{eq:f_samp_bd}
 \abs{\ftil(x_i) + q^* - f(x_i)} \leq \frac{\delta(n)}{4}, \quad \forall i \in [n].   
\end{equation}

Next, we use these estimates to obtain an estimate of $f$ via spline-based quasi-interpolants (QI). QI's are linear operators $\calQ_n: C([0,1]) \rightarrow C([0,1])$, where $\calQ_n(g)$ depends only on the values $g(x_i)$ for $i = 1,\dots,n$. These objects are classical in the literature, see for e.g. \cite{devore_book, deBoor1990} for a detailed overview including the construction of these operators. It is well known (see \cite[Remark 2.6]{fanuel20}) that there exists a constant $C_{l,\alpha,M} > 0$ (depending only on $l,\alpha,M$) such that
\begin{equation} \label{eq:qi_bound}
 \norm{\calQ_n(g) - g}_{\infty} \leq C_{l,\alpha} n^{-(l+\alpha)} = C_{l,\alpha} n^{-\beta}, 
\end{equation}
for all $ g \in \holdclass$.
Denoting $\ftil \in C([0,1])$ to be a function which takes the values $\ftil(x_i)$ for each $i$, the estimate $\est{f}$ of $f$ is obtained as $\est{f} = \calQ_n(\ftil)$. The complete procedure is outlined as Algorithm \ref{algo:Main} below.   

\input{algorithm}

\paragraph*{Main result for recovering $f$} The following theorem provides a $L_{\infty}$ error bound for recovering $f \in \holdclass$ up to a global integer shift. The proof follows in the same manner as that of \cite[Theorem 2.4]{fanuel20} and is omitted.
\begin{theorem} \label{thm:main_rec_f}
Under the notations in Theorem \ref{thm:main_denoise_mod}, for any $f \in \holdclass$ with $\beta:= l + \alpha$, suppose that $f, n, b$ and the kernel $K: [-1,1] \rightarrow \matR$ satisfy the conditions of Theorem \ref{thm:main_denoise_mod}. 
Recalling $\delta(n)$ from \eqref{eq:h_denoise_fin_bd} which is defined for a given constant $c \geq 2$, suppose additionally that 
$\delta(n) + \frac{2L}{n^{\min\set{\beta,1}}} < 1$ with $L$ as in~\eqref{eq:lip_L}. If the QI operator $\calQ_n$ satisfies \eqref{eq:qi_bound}, then with probability at least $1-4n^{1-c}$, it holds that the estimate $\est{f} = \calQ_n(\ftil)$ satisfies (for some integer $q^*$) the bound
\begin{equation*}
    \norm{\est{f} + q^* - f}_{\infty} \leq C_{l,\alpha,M} n^{-\beta} + C \delta(n).
\end{equation*}
Here $C > 0$ is an absolute constant while $C_{l,\alpha,M}$ is the constant in \eqref{eq:qi_bound}.
\end{theorem}
Theorem \ref{thm:main_rec_f} suggests that 
$\|\est{f} + q^* - f\|_{\infty} = O\left(\left(\log (n)/n \right)^{\frac{\beta}{2\beta + 1}}\right).$ 
As remarked earlier, this rate matches the optimal $L_{\infty}$ rate for estimating functions lying in $\holdclass$, for the model \eqref{eq:fmod1_model} without the modulo operation (see \cite{nemirovski2000topics}).

%% file: algorithm.tex
\begin{minipage}{0.9\linewidth}
\centering
\begin{algorithm}[H]
    \caption{Function recovery from noisy modulo samples} 
    \label{algo:Main} 
    \begin{algorithmic}[1] 
    \STATE \textbf{Input:} noisy modulo samples $y_i \in [0,1)$ for each $x_i = i/n$, $i=1,\dots,n$; $l \in \mathbb{N}_0$.
    
    \STATE \textbf{Output:} Estimate $\est{f}:[0,1] \rightarrow \matR$ of $f$.
    
    \STATE {\bf Denoising step:}  Form $z_i = \exp(\iota 2\pi y_i)$ for $i=1,\dots,n$.
    
    \FOR{$i \in [n]$}
    \STATE Compute LP$(l)$ estimators $\tilde{h}(x_i) \in \mathbb{C}$ as in \eqref{eq:htil_lp_est}.
    
    \STATE Obtain estimate $\est{g}(x_i) = \est{f(x_i) \bmod 1}$ as in \eqref{eq:hest_proj}.
    \ENDFOR 
    
    \STATE {\bf Unwrapping step:} Use $(\est{g}(x_i))_{i \in [n]}$ to obtain $(\ftil(x_{i}))_{i \in [n]}$ as in \eqref{eq:seq_unwrap_proc}.
    
    \STATE {\bf Recovering $f$:} Output $\est{f} = \calQ_n(\ftil)$ where $\calQ_n$ is a spline-based quasi interpolant operator.
    \end{algorithmic}
\end{algorithm}
\vspace{0.1cm}
\end{minipage}

%% file: simulations.tex
\section{Numerical simulations}
We consider the function
\begin{equation}
f(x) = 4x \cos(2\pi x)^2 - 2 \sin(2\pi x)^2 + 4.7\label{eq:function}
\end{equation}
and the noise model~\eqref{eq:fmod1_model} with $\sigma = 0.12$. The output of Algorithm~\ref{algo:Main} is illustrated in Figure~\ref{fig:ex2}. The following other methods are considered: kNN denoising~\cite{fanuel20} (kNN), an unconstrained quadratic program~\cite{tyagi20} (UCQP) and a trust region subproblem~\cite{CMT18_long} (TRS).
A comparison between these methods is given in Figure~\ref{fig:wraparound_RMSE_ex2}, where the Root Mean Square Error (RMSE) between the ground truth and the recovered samples is displayed\footnote{The code is available at \url{https://github.com/mrfanuel/denoising-modulo-samples-local-polynomial-estimator}}.
\begin{figure}
    \includegraphics[scale = 0.55]{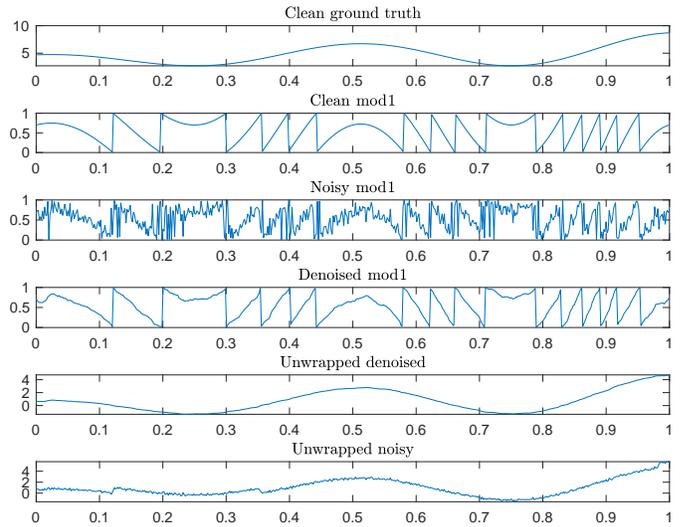}
    \caption{Recovery of the function~\eqref{eq:function} from $n=600$ noisy modulo samples via a local polynomial estimator. \label{fig:ex2}}
\end{figure}
\begin{figure}
    \includegraphics[scale = 0.25]{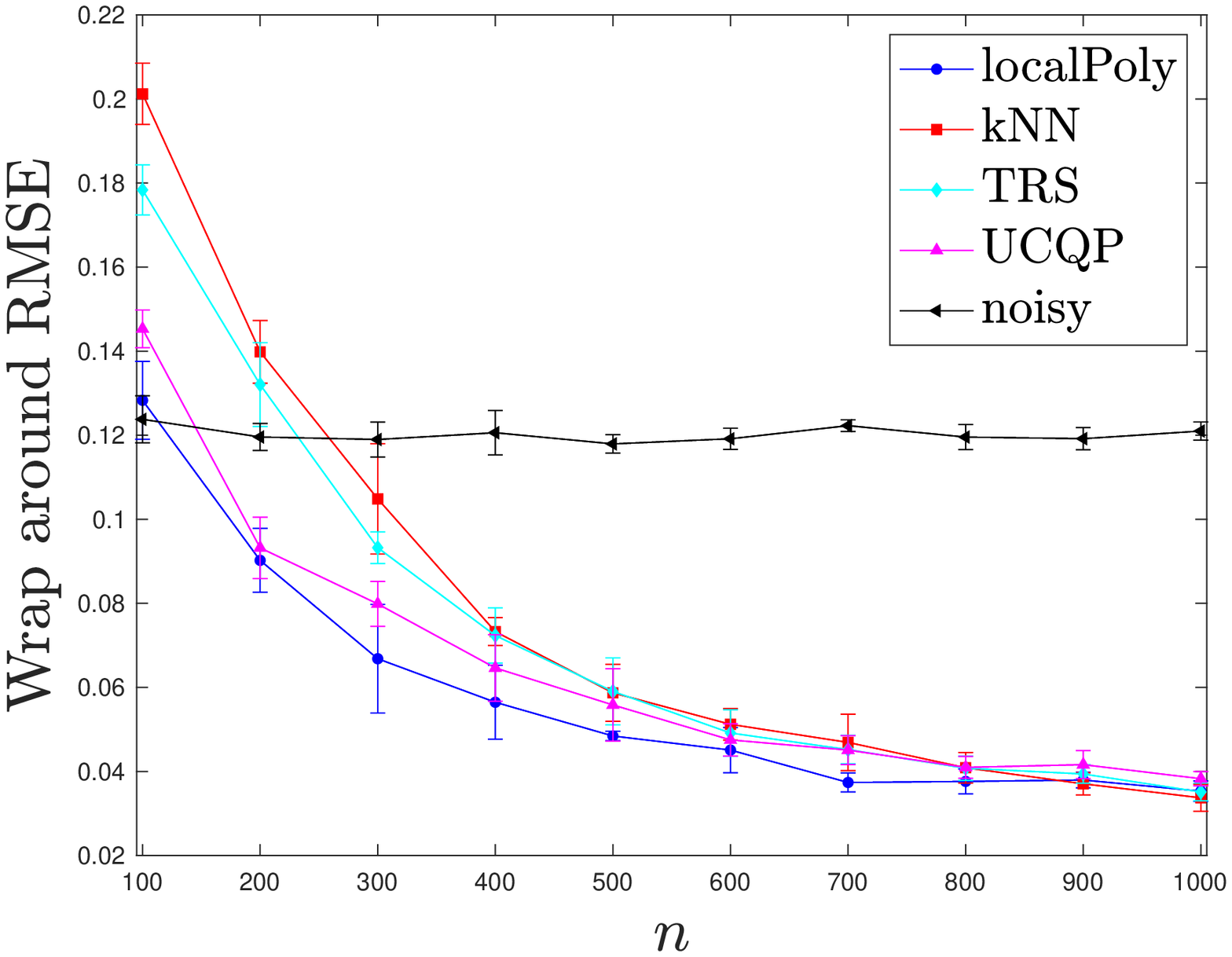}
    \includegraphics[scale = 0.25]{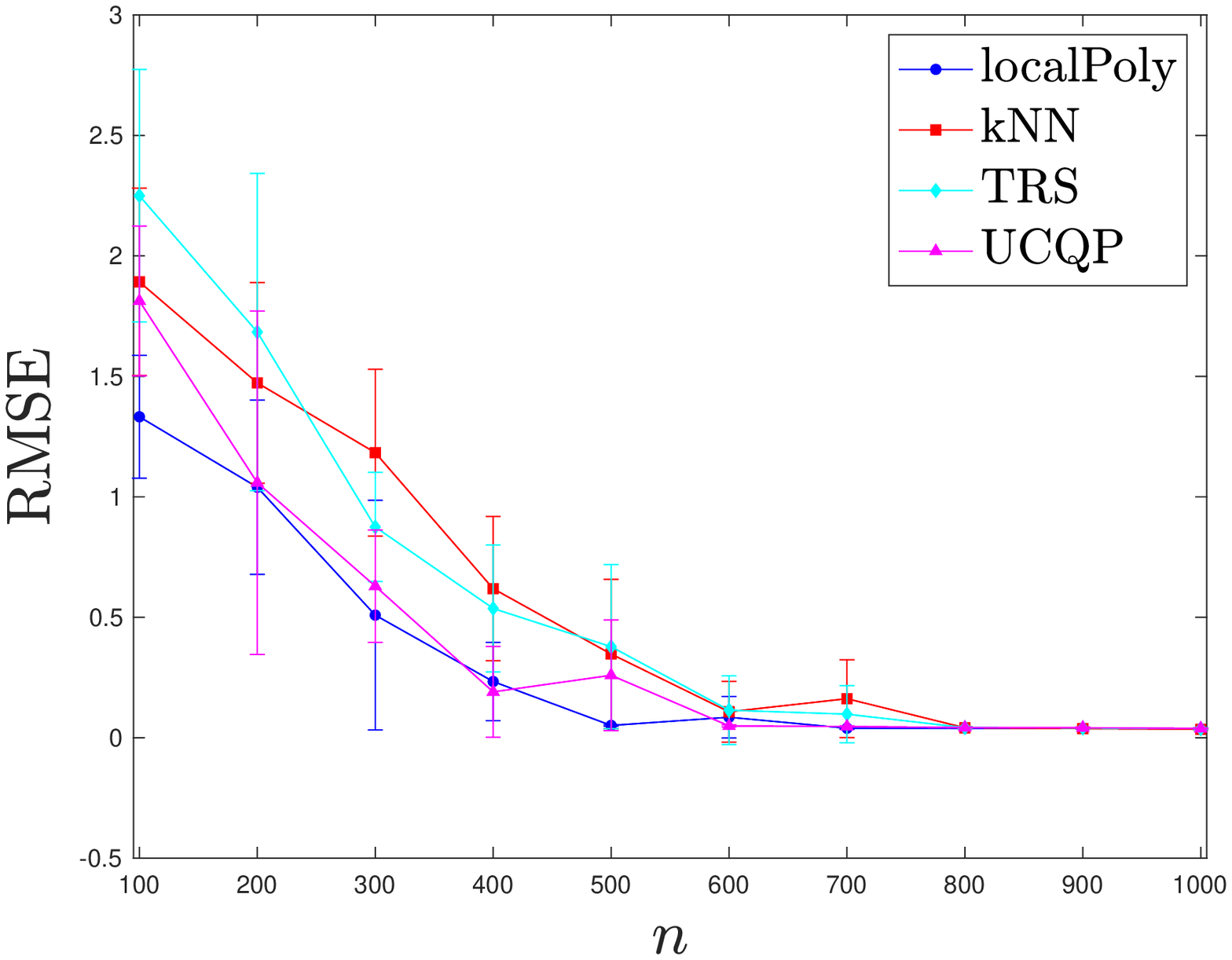}
    \caption{LHS: wrap around RMSE between denoised modulo samples and ground truth vs number of samples $n$. RHS: RMSE between the unwrapped denoised modulo samples and ground truth samples vs $n$. The markers denote the average over $5$ runs and the errors are standard deviations.\label{fig:wraparound_RMSE_ex2}}
\end{figure}
We choose the parameters in the following way.
For the local polynomial estimator, we take  $l = 2$, $\beta = 2.4$ and
$b = 0.1(\log(n)/n)^{\beta/(2\beta+1)}$.
The number of neighbours of kNN is  $k = \ceil{k^\star}$ with $k^\star = 0.09 n^{\frac{2}{3}}\left( \log n\right)^{\frac{1}{3}}$; see~\cite{fanuel20}.
We follow  the analysis of (UCQP) and (TRS) by Tyagi \cite[Corollary $4$ and Corollary $8$]{tyagi20} which gives $\lambda \asymp \left( n^{10/3}\right)^{1/4}$ up to multiplicative constants. Hence, we choose $\lambda= 0.04 n^{10/12}$. The latter methods use a graph $G = ([n], E)$ which is simply here a path graph
$E = \set{\set{i,i+1}: i = 1,\dots,n-1}$.
From the numerical results in Figure~\ref{fig:wraparound_RMSE_ex2}, we observe that all the methods have a similar behavior given the error bars, whereas we conjecture that a fine-tuning of the parameters could boost the performance for a small sample size $n$.

%% file: holder_smoothness.tex





%
  \appendices
\section{Proof of Lemma~\ref{lem:denoise_bias_var}}
Starting with the observation that for any $x \in [0,1]$, 
$    \est{h}(x) = \htil(x)/|\htil(x)| = \expp{2}\htil(x)/|\expp{2}\htil(x)|,$
we obtain using \cite[Proposition 3.3]{LiuPhaseSynchronization} the bound 
\begin{equation} \label{eq:proj_ineq}
    \abs{\est{h}(x) - h(x)} \leq 2\abs{\expp{2}\htil(x) - h(x)}.
\end{equation}
Henceforth, we will focus on bounding the quantity $\abs{\expp{2}\htil(x) - h(x)}$. Since $B_{nx} \succ 0$ for all $x \in [0,1]$, we know from Proposition \ref{prop:wst_reprod} that  $\sum_{i=1}^n W^*_{ni}(x) = 1$, and so using \eqref{eq:htil_uniq_exp}, we can write
\begin{align*}
\expp{2}\htil(x) - h(x) = \underbrace{\sum_i(\expp{2} z_i - h_i) \Wst_{ni}(x)}_{V(x)} + \underbrace{\sum_i(h_i - h(x)) \Wst_{ni}(x)}_{\bias(x)}.
\end{align*}
$V(x)$ is the `variance' term and is the sum of centered, independent random variables (indeed, $\expec[z_i] = \expn{2} h_i$ for each $i$). $\bias(x)$ denotes the `bias' of the estimator. In what follows, we will show that for any given $x \in [0,1]$, 
\begin{align} \label{eq:var_bound_x}
\abs{\bias(x)} \leq  \frac{2M' C_* b^{\beta}}{l!}, \text{ and }  \prob\left(\abs{V(x)} \geq 4c C_* A(\sigma) \left(\frac{\log n }{bn} \right)^{1/2} \right) \leq 4n^{-c}
\end{align}
which together with \eqref{eq:proj_ineq} readily yields the statement of the lemma after a union bound (with $x = x_i$) over $i \in [n]$.
%
%
\paragraph*{Bounding $\abs{\bias(x)}$}
We start by writing 
\begin{equation*}
    \bias(x) = \sum_{i=1}^n (h_{R,i} - h_R(x) + \iota  (h_{I,i} - h_I(x)) \Wst_{ni}(x).
\end{equation*}
Since $h_R, h_I \in \holdclassnew$, hence we have using Proposition \ref{prop:wst_reprod} that
\begin{align*}
    \bias(x)
     = \sum_{i} \frac{(h_R^{(l)}(x+\tau_i(x_i - x)) -h_R^{(l)}(x))}{l!} (x_i - x)^l\Wst_{ni}(x) + \iota \sum_{i} \frac{(h_I^{(l)}(x+\tau_i'(x_i - x)) -h_I^{(l)}(x))}{l!} (x_i - x)^l\Wst_{ni}(x)
\end{align*}
where $\tau_i,\tau_i' \in [0,1]$ for all $i \in [n]$. Now we can bound $\abs{\bias(x)}$ as
\begin{align*}
    \abs{\bias(x)} &\leq 2M'\sum_i \frac{\abs{x_i-x}^{\beta}}{l!} \abs{\Wst_{ni}(x)} = 2M'\sum_i \frac{\abs{x_i-x}^{\beta}}{l!} \abs{\Wst_{ni}(x)} \indic_{\set{\abs{x_i-x} \leq b}} \leq \frac{2M' C_* b^{\beta}}{l!}
\end{align*}
where we used Lemma \ref{lem:wst_props}.

\paragraph*{Bounding $\abs{V(x)}$}
We start by writing 
\begin{align*}
V(x) &= \sum_i(\expp{2} z_{R,i} - h_{R,i}) \Wst_{ni}(x) + \iota \sum_i(\expp{2} z_{I,i} - h_{I,i}) \Wst_{ni}(x) =: V_{R}(x) + \iota V_{I}(x).
\end{align*}
Since $\abs{V(x)} \leq \abs{V_R(x)} + \abs{V_I(x)}$, we will bound the two RHS terms using Bernstein's inequality. Its enough to bound $\abs{V_R(x)}$ since the same bound will apply for $\abs{V_I(x)}$. To this end, first note that each term in the summation is uniformly bounded as 
\begin{equation*}
 \abs{\expp{2} z_{R,i} - h_{R,i}} \  \abs{\Wst_{ni}(x)}  
 \leq (1+\expp{2})\frac{C_*}{nb},
\end{equation*}
where we used Lemma \ref{lem:wst_props}. Moreover, it is easy to verify the bound 
\begin{align*}
    \expec\left[(\expp{2} z_{R,i} - h_{R,i})^2 \right] 
    \leq \expec \left[\abs{\expp{2} z_{i} - h_{i}}^2 \right] = \expp{4} \expec \left[\abs{z_i - \expn{2} h_i}^2 \right] = \expp{4}-1.
\end{align*}
Hence, it follows using Lemma \ref{lem:wst_props} that 
\begin{align*}
  \sum_i \expec[(\expp{2} z_{R,i} - h_{R,i})^2] (\Wst_{ni}(x))^2  
  &\leq (\expp{4}-1)\sum_i(\Wst_{ni}(x))^2 \\
  &= (\expp{4}-1)\sum_i(\Wst_{ni}(x))^2 \indic_{\set{\abs{x_i-x} \leq b}}\\
  &\leq (\expp{4}-1)\frac{C_*^2}{n^2 b^2} \sum_i \indic_{\set{\abs{x_i-x} \leq b}} \\
  &\leq (\expp{4}-1)\frac{C_*^2}{n^2 b^2} \max\set{2nb,1} \\
  &= 2(\expp{4}-1)\frac{C_*^2}{n b} \quad (\text{ if } 2nb \geq 1).
\end{align*}
Then, by using Bernstein's inequality, we obtain for any $c \geq 2$ that with probability at least $1-2n^{-c}$, 
\begin{align*}
 \abs{V_R(x)} 
 &\leq \sqrt{4c \log n(\expp{4}-1)\frac{C_*^2}{n b}} + \frac{2c C_*}{3}\left(\frac{\log n}{nb}\right)(1+\expp{2}) \\
 &\leq 2cC_*A(\sigma) \left(\frac{\log n}{nb}\right)^{1/2} \quad (\text{ if } nb \geq \log n).
\end{align*}
This leads to the bound on $\abs{V(x)}$ in \eqref{eq:var_bound_x} after a union bound, which also concludes the proof.

\section{Proof of Proposition \ref{prop:smoothness_h} \label{sec:proof_smoothness_h}} 

Define for convenience $\phi_s(x) = \cos(x - s\pi/2)$ with $s\in\{0,1\}$ so that $\phi_0(x) = \cos(x)$ and $\phi_1(x) = \sin(x)$.
We now consider functions of the form
\[
  \phi_s(g(x)) \text{ where } g(x) = 2\pi f(x) \text{ with } f \in\holdclass,
\]
so that $g\in \holdclassTwoPi$. \MF{The boundedness assumption on $\norm{f^{(\ell)}}_{\infty}$ implies $\norm{g^{(\ell)}}_{\infty} \leq 2\pi \kappa$ for all $0 \leq \ell \leq l$}. In particular, we have $\phi_0(g(x)) = h_R(x)$ and $\phi_1(g(x)) = h_I(x)$. For simplicity, we omit the dependence on $s$ and simply work with $\phi(x)$.

Our objective is to show that there exists a constant $M'>0$ such that 
\[
|\phi^{(l)}(g(x)) - \phi^{(l)}(g(y))| \leq M' |x-y|^{\beta - l},
\]
for all $x,y\in [0,1]$.
\subsection*{Preliminary results \label{prelim}}
We start by stating useful elementary results.
First, we give a simple result based on Taylor's theorem.
\begin{lemma}\label{eq:lemDerivative}
  Let $g\in \holdclassTwoPi$ and $\beta = l + \alpha$. Assume $l\geq 1$. Then, for all $0\leq p < l$, the following inequality
  \[
    |g^{(p)}(x) - g^{(p)}(y)| \leq \frac{2\pi M}{(l-p)!} |x-y|^{\beta - p} + \sum_{k=1}^{l -p} \frac{|g^{(k+p)}(y)|}{k!} |x-y|^k,
  \]
  holds for all $x,y\in [0,1]$.
\end{lemma}
\begin{proof}
  By using the order $q-1$ Taylor expansion of a $q$-times differentiable function $z(x)$ at $y$ and the remainder theorem, we have
  \[
    z(x) = \sum_{k=0}^q \frac{z^{(k)}(y)}{k!}(x-y)^k + \frac{(x-y)^{q}}{q!} \left( z^{(q)}(\xi) - z^{(q)}(y)\right),
  \]
  for some $\xi$ between $x$ and $y$, thanks to the remainder theorem. Now choose $z(x) = g^{(p)}(x)$ and $q = l -p$. This gives
  \[
    g^{(p)}(x) = \sum_{k=0}^{l -p} \frac{g^{(p+k)}(y)}{k!}(x-y)^k + \frac{(x-y)^{l -p}}{(l -p)!} \left( g^{(l)}(\xi) - g^{(l)}(y) \right).
  \]
  Hence, by using the definition of the  H\"older class, we find
  \begin{align*}
    \left|g^{(p)}(x)-g^{(p)}(y) - \sum_{k=1}^{l -p} \frac{g^{(p+k)}(y)}{k!}(x-y)^k\right| \leq \frac{2\pi M}{(l-p)!}|x-y|^{l-p} |\xi -y|^{\beta -l} \leq \frac{2\pi M}{(l-p)!}|x-y|^{\beta-p},
  \end{align*}
 where we used $|\xi -y|\leq |x - y|$ since $\xi$ is between $x$ and $y$.
 Next, we use the reverse triangle inequality to yield
 \begin{align*}
  \left||g^{(p)}(x)-g^{(p)}(y)| - \big|\sum_{k=1}^{l -p} \frac{g^{(p+k)}(y)}{k!}(x-y)^k\big|\right| \leq  \frac{2\pi M}{(l-p)!}|x-y|^{\beta-p}.
\end{align*}
Hence, it holds
\begin{align*}
  \left|g^{(p)}(x)-g^{(p)}(y)\right| - \sum_{k=1}^{l -p} \left|\frac{g^{(p+k)}(y)}{k!}(x-y)^k\right| \leq  \frac{2\pi M}{(l-p)!}|x-y|^{\beta-p},
\end{align*}
and the desired result follows by using once again the triangle inequality.
\end{proof}
A simple consequence of Lemma~\ref{eq:lemDerivative} is given below.
\begin{lemma} \label{eq:boundsOnDerivatives}
    Let $g\in \holdclassTwoPi$ and denote $\beta = l + \alpha$. Suppose $l\geq 1$ and let $p$ be an integer such that $0\leq p < l$. \MF{Let $\kappa>0$ a constant such that $|g^{(\ell)}(x)|\leq 2\pi \kappa$ for all $x\in[0,1]$ and $0 \leq \ell \leq l$}. Then, there exists a constant $C_2(M,p,\kappa)>0$ such that
    \[
      |g^{(p)}(x) - g^{(p)}(y)| \leq C_2 |x-y|^{\beta - l}.
    \]
\end{lemma}
\begin{proof}
   Recall  the following fact: if $x,y\in [0,1]$, then  $|x-y|\leq |x-y|^s$ for all $s\in (0,1]$. Then, the result follows by using Lemma~\ref{eq:lemDerivative} and \MF{$|g^{(\ell)}(y)|\leq 2\pi\kappa$} for all $y\in [0,1]$ and for all $p+1\leq \ell \leq l$.

\end{proof}
The following fact is a property of the sine function.
\begin{fact} Let $p\in \mathbb{N}_0$. Then, it holds
  $
  |\phi^{(p)}(x)| \leq 1
  $
  for all $x\in [0,1]$.
\end{fact}
Another useful identity in given in Fact~\ref{fact:diff} which is a simple algebraic fact that is readily proved by adding and subtracting identical terms.
\begin{fact}\label{fact:diff}
  Let $a_1, \dots, a_n$ and $b_1, \dots, b_n$ be real numbers with $n\geq 2$. Then, we have
  \[
    \prod_{i=1}^{n} a_i - \prod_{i=1}^{n} b_i = (a_n-b_n)\prod_{i=1}^{n-1} a_i + \sum_{\ell = 1}^{n-2}\left((a_{n-\ell} - b_{n-\ell}) \prod_{i = n-\ell +1 }^{n} b_i \prod_{j=1}^{n-\ell-1} a_j\right) + (a_1-b_1) \prod_{i=2}^{n} b_i 
  \]
\end{fact}
The next statement is standard about polynomial factorization.
\begin{fact}\label{fact:factorization}
Let $a,b\in \mathbb{R}$. For all integer positive $n\geq 1$, there exists a polynomial $p_{n-1}(a,b)$ of order $n-1$ such that
\[
  a^n - b^n = (a-b) p_{n-1}(a,b).
\]
\end{fact}
Finally, we recall the well-known Fa\`a di Bruno formula for derivatives of composed functions.
\begin{fact}[Fa\`a di Bruno formula]\label{fact:FaaDiBruno}
Let $l\in \mathbb{N}_0$ and let $g,\bar{g}$ be real valued $l$-times differentiable functions. We have 
\[
(\bar{g}\circ g)^{(l)}(x) = \sum_{k_1,\dots,k_l}\frac{l!}{k_1! \dots k_l!} \bar{g}^{(k)}(g(x)) \left(\frac{g^{(1)}(x)}{1!}\right)^{k_1} \dots \left(\frac{g^{(l)}(x)}{l!}\right)^{k_l},
\]
with $k = k_1 + \dots + k_l$ and where the sum goes over $k_1,\dots,k_l$ such that $k_1 + 2k_2 + \dots + l k_l = l$.
\end{fact}
\subsection*{Main part of the proof of  Proposition~\ref{prop:smoothness_h} \label{mainpart}}
The proof mainly relies on the triangle inequality applied to the Fa\`a di Bruno formula (see, Fact~\ref{fact:FaaDiBruno} with $\bar{g}= \phi$) and Fact~\ref{fact:diff}. 
\MF{Let us also recall that $|g^{(\ell)}(x)|\leq 2\pi\kappa$ for all $x\in[0,1]$ and $0\leq \ell \leq l$}. 
Using Fa\`a di Bruno formula with the triangle inequality, we obtain
\begin{equation}
        \left|(\phi\circ g)^{(l)}(x) - (\phi\circ g)^{(l)}(y)\right|\leq  \sum_{k_1,\dots,k_l}\frac{l!}{k_1! \dots k_l!}  \left|\Delta_{k_1,\dots,k_l}(x)-\Delta_{k_1,\dots,k_l}(y)\right|,\label{eq:BoundByFaaDiBruno}
\end{equation}
with $k = k_1 + \dots + k_l$ and where the sum goes over $k_1,\dots,k_l$ such that $k_1 + 2k_2 + \dots + l k_l = l$. Here, we defined
\[
\Delta_{k_1,\dots,k_l}(x) = \phi^{(k)}(g(x)) \left(\frac{g^{(1)}(x)}{1!}\right)^{k_1} \dots \left(\frac{g^{(l)}(x)}{l!}\right)^{k_l}.
\]
Notice that the above product includes at most $l+1$ non trivial factors.
Next, we use Fact~\ref{fact:diff} and obtain
\begin{align*}
    &\Delta_{k_1,\dots,k_l}(x) - \Delta_{k_1,\dots,k_l}(y)\\ &=\left(\left(\frac{g^{(l)}(x)}{l!}\right)^{k_l} - \left(\frac{g^{(l)}(y)}{l!}\right)^{k_l}\right) \phi^{(k)}(g(x)) \left(\frac{g^{(1)}(x)}{1!}\right)^{k_1} \dots \left(\frac{g^{(l-1)}(x)}{(l-1)!}\right)^{k_{l-1}}\\
    &+ \sum_{\ell = 1}^{l-1}\left(\left(\frac{g^{(l-\ell)}(x)}{(l-\ell)!}\right)^{k_{l-\ell}}-\left(\frac{g^{(l-\ell)}(y)}{(l-\ell)!}\right)^{k_{l-\ell}}\right) \prod_{i = l-\ell +1 }^{l} \left(\frac{g^{(i)}(y)}{i!}\right)^{k_i} \prod_{j=1}^{l-\ell-1} \left(\frac{g^{(j)}(y)}{j!}\right)^{k_j}\\
    &+ \left(\phi^{(k)}(g(x))-\phi^{(k)}(g(y))\right) \left(\frac{g^{(1)}(x)}{1!}\right)^{k_1} \dots \left(\frac{g^{(l)}(x)}{l!}\right)^{k_l}.
\end{align*}
A triangle inequality and the bound \MF{$|g^{(\ell)}(x)|\leq 2 \pi \kappa$ for all $x\in[0,1]$ and $0\leq \ell \leq l$} yields
\begin{align*}
    |\Delta_{k_1,\dots,k_l}(x) - \Delta_{k_1,\dots,k_l}(y)| & \leq  \left|\left(\frac{g^{(l)}(x)}{l!}\right)^{k_l} - \left(\frac{g^{(l)}(y)}{l!}\right)^{k_l}\right| \left(\frac{\MF{2\pi\kappa}}{1!}\right)^{k_1} \dots \left(\frac{\MF{2\pi\kappa}}{(l-1)!}\right)^{k_{l-1}}\\
    &+ \sum_{\ell = 1}^{l-1}\left|\left(\frac{g^{(l-\ell)}(x)}{(l-\ell)!}\right)^{k_{l-\ell}}-\left(\frac{g^{(l-\ell)}(y)}{(l-\ell)!}\right)^{k_{l-\ell}}\right| \prod_{i = l-\ell +1 }^{l} \left(\frac{\MF{2\pi\kappa}}{i!}\right)^{k_i} \prod_{j=1}^{l-\ell-1} \left(\frac{\MF{2\pi\kappa}}{j!}\right)^{k_j}\\
    &+ \left|\phi^{(k)}(g(x))-\phi^{(k)}(g(y))\right| \left(\frac{\MF{2\pi\kappa}}{1!}\right)^{k_1} \dots \left(\frac{\MF{2\pi\kappa}}{l!}\right)^{k_l}.
\end{align*}
Consider each of the three terms on the RHS of the last inequality.

For the first term, we can assume that $k_l \geq 1$, otherwise it vanishes. By using the factorization result in Fact~\ref{fact:factorization}, we have 
\begin{align*}
\left|\left(g^{(l)}(x)\right)^{k_l} - \left(g^{(l)}(y)\right)^{k_l}\right| &= \left|g^{(l)}(x) - g^{(l)}(y)\right| \cdot \left|p_{k_l-1}(g^{(l)}(x),g^{(l)}(y))\right|\\
&\leq  \left|g^{(l)}(x) - g^{(l)}(y)\right| \bar{C}_0(\kappa,l)\\
& \leq 2\pi M |x-y|^{\beta -l} \bar{C}_0(\kappa,l),
\end{align*}
where we used a triangle inequality to bound $\left|p_{k_l-1}(g^{(l)}(x),g^{(l)}(y))\right|$ and where $\bar{C}_0(\kappa,l)$ is some positive polynomial of $\kappa,l$.

For the second term, in the same way, we can upper bound each of the terms in the finite sum, i.e., for a given $\ell$,
\[
    \left|\left(g^{(l-\ell)}(x)\right)^{k_{l-\ell}}-\left(g^{(l-\ell)}(y)\right)^{k_{l-\ell}}\right|
    \leq 
    \bar{C}_1(\kappa,l) \left|g^{(l-\ell)}(x)-g^{(l-\ell)}(y)\right|\leq 
    \bar{C}_1(\kappa,l) C_2(M,\kappa,l-\ell)|x-y|^{\beta -l},
\]
where we used Lemma~\ref{eq:boundsOnDerivatives} for the last inequality. Here, $\bar{C}_1(\kappa,l)$ is some positive polynomial of $\kappa,l$ and $C_2(M,\kappa,l-\ell) > 0$ depends only on the indicated terms.

For the third term, remark that  $|\phi(x) -\phi(y)| \leq  |x-y|$ for all $x,y\in \mathbb{R}$. By definition of $\phi(x)$, this implies that \[\left|\phi^{(k)}(g(x))-\phi^{(k)}(g(y))\right|\leq  \left|g(x)-g(y)\right|,\]
for any non-negative integer $k$. Also, by Lemma~\ref{eq:boundsOnDerivatives}, $\left|g(x)-g(y)\right| \leq C_2(M,\kappa) |x-y|^{\beta-l}$, which yields
\[
    \left|\phi^{(k)}(g(x))-\phi^{(k)}(g(y))\right|\leq C_2(M,\kappa) |x-y|^{\beta-l},
\]
for all $x,y\in[0,1]$ and all non-negative integer $k$. 

Finally, putting everything together, we obtain
\begin{align*}
    |\Delta_{k_1,\dots,k_l}(x) - \Delta_{k_1,\dots,k_l}(y)|  \leq C(M,\kappa,l)|x-y|^{\beta-l},
\end{align*}
where $C(M,\kappa,l) > 0$ depends only on $M,\kappa$ and $l$. Plugging this within~\eqref{eq:BoundByFaaDiBruno}, we obtain the result.